\documentclass[12pt]{amsart}
\usepackage{amsmath,amssymb,amsthm}
\usepackage{esint}
\usepackage[margin=1in]{geometry}
\usepackage{microtype}
\usepackage[hidelinks]{hyperref}
\hypersetup{pdftitle={Large-volume fractional isoperimetry on products of closed manifolds and Euclidean spaces},pdfauthor={Juneyoung Seo}}
\numberwithin{equation}{section}
\theoremstyle{plain}
\newtheorem{theorem}{Theorem}[section]
\newtheorem{lemma}[theorem]{Lemma}
\newtheorem{proposition}[theorem]{Proposition}
\newtheorem{corollary}[theorem]{Corollary}
\newcommand{\R}{\mathbb R}
\newcommand{\dd}{\,d}

\title[Large-volume fractional isoperimetry]{Large-volume fractional isoperimetry on products of closed manifolds and Euclidean spaces}
\author[J. Seo]{Juneyoung Seo}
\address{Innovation Center for Atomic Science, DAU G-LAMP Project Group, Dong-A University, Busan 49315, Republic of Korea}
\email{juneys.dau@gmail.com}
\thanks{This research was supported by Global - Learning \& Academic research institution for Master's$\cdot$PhD students, and Postdocs (LAMP) Program of the National Research Foundation of Korea (NRF) grant funded by the Ministry of Education (RS-2025-25440216).}

\theoremstyle{definition}
\newtheorem{definition}[theorem]{Definition}
\newcommand{\cP}{\mathcal{P}}

\theoremstyle{plain}
\newtheorem{citedtheorem}[theorem]{Theorem}
\begin{document}
\begin{abstract}
Let $M$ be a closed connected Riemannian manifold and let $\mathcal P_s$ be
the fractional perimeter on $M\times\mathbb R^k$ defined through its heat
semigroup, where $k\ge1$ and $0<s<1$. We prove that the straight cylinder
$M\times B_R$, with $B_R\subset\mathbb R^k$ a ball of radius $R$, uniquely
minimizes $\mathcal P_s$ among all measurable sets of the same volume,
up to translations in $\mathbb R^k$ and null sets, whenever $R\sqrt{\lambda_1(M)}$
is at least an explicit threshold depending only on $k$ and $s$. Here $\lambda_1(M)$ is the first positive eigenvalue of the
nonnegative Laplace--Beltrami operator $-\Delta_M$ on $M$. Under a larger threshold, the perimeter
deficit controls both the deviation of the slice volumes from their mean and
the square of the asymmetry relative to straight cylinders. For flat tori this gives large-volume uniqueness of straight cylinders for the periodic fractional perimeter, including singly periodic sets in every ambient dimension.
The proof combines spectral reduction on $M$ with rearrangement of the
Euclidean slices and a one-variable comparison between an interaction gain and a concavity loss.
\end{abstract}
\maketitle
\section{Introduction}
Let $(M,g)$ be a closed connected Riemannian
manifold of dimension $m\ge1$, let $k\ge1$, and equip
$N=M\times\mathbb R^k$ with its product metric. For the classical
perimeter, the isoperimetric regions of sufficiently large volume in
$N$ are the \emph{straight cylinders} $M\times B_R(z_0)$
\cite{Gonzalo,RV}, where $B_R(z_0)\subset\mathbb R^k$ is the ball of radius
$R$ centered at $z_0$. We establish the corresponding uniqueness and
quantitative stability for the heat-kernel fractional perimeter with
$0<s<1$, under explicit lower bounds on $R\sqrt{\lambda_1(M)}$ that depend
only on $k$ and $s$; here $\lambda_1(M)$ is the first positive eigenvalue of $-\Delta_M$.
When $M$ is a flat torus, this resolves the large-volume uniqueness
question for the periodic fractional perimeter in every ambient dimension.

\subsection*{Setting and main results}
Let $\mu$ be the Riemannian volume measure of $(M,g)$, write $|M|=\mu(M)$,
and equip $N$ with the product
of $\mu$ and Lebesgue measure. We write $|\cdot|$ for the measure of a set in its
ambient space ($M$, $\mathbb R^d$ or $N$; surface measure for spheres), points of
$N$ as $X=(x,z)$ and $Y=(y,w)$, and $dX=d\mu(x)\,dz$. Put $n=m+k$. Let $\Delta_M=\operatorname{div}_g\nabla_g$ be the Laplace--Beltrami
operator of $M$ and let $h_t(x,y)$, $t>0$, be the heat
kernel of $M$, that is, the integral kernel of $e^{t\Delta_M}$ with respect to
$\mu$. The heat kernel of $N$ is then
\[
h_t^N(X,Y)=h_t(x,y)\,G_t(z-w),\qquad G_t(z)=(4\pi t)^{-k/2}e^{-|z|^2/(4t)}.
\]

For measurable $E\subset N$ of finite measure define
\begin{equation}\label{eq:def}
\mathcal P_s(E)=\kappa_{n,s}\int_0^\infty\int_E
\int_{N\setminus E}h_t^N(X,Y)\,dY\,dX\,\frac{dt}{t^{1+s/2}},
\end{equation}
where $\kappa_{n,s}=\pi^{n/2}/(2^s\Gamma((n+s)/2))$ for every integer $n\ge1$.
This heat-kernel definition follows the construction of
\cite[Definition~3.1 and Section~3.2]{CFS2} on the complete product $N$.
The normalization agrees with the Euclidean kernel
$|X-Y|^{-n-s}$ \cite{CRS}. On a general product the kernel is determined
by the heat semigroup rather than by a power of geodesic distance.
The perimeter is allowed to be infinite.

Write $E_x=\{z:(x,z)\in E\}$ for the slice at $x$, and let $\omega_k$ be the volume of the unit ball of $\mathbb R^k$. For $0<|E|<\infty$ put $A=|E|/|M|$ and $R=(A/\omega_k)^{1/k}$, so $|E|=|M|\omega_kR^k$ and the cylinder has the same volume as $E$.

The thresholds $T_{k,s}(1)$ and $T_{k,s}(2)$, depending only on $k,s$, are defined in Definition~\ref{def:rho}.
\begin{theorem}\label{thm:A}
Let $(M,g)$ be a closed connected Riemannian manifold, let $k\ge1$
and $0<s<1$, and let $E\subset M\times\mathbb R^k$ be measurable with
$|E|=|M|A$, $0<A<\infty$. Put $R=(A/\omega_k)^{1/k}$.
If $R\sqrt{\lambda_1(M)}\ge T_{k,s}(1)$, then
\[
\mathcal P_s(E)\ge\mathcal P_s(M\times B_R).
\]
Equality holds if and only if $E$ agrees, up to a null set, with
$M\times B_R(z_0)$ for some $z_0\in\mathbb R^k$.
\end{theorem}

We measure the distance from $E$ to straight cylinders by the \emph{asymmetry}
\[
\alpha(E)=\inf_{z_0\in\mathbb R^k}
 \frac{|E\triangle(M\times B_R(z_0))|}{|E|}\in[0,2).
\]

\begin{theorem}\label{thm:B}
In the setting of Theorem~\ref{thm:A}, suppose
$R\sqrt{\lambda_1(M)}\ge T_{k,s}(2)$. Then, with the positive coefficients
$C_1$ and $C_2$ defined in \eqref{eq:Cchoice},
\[
\frac{\mathcal P_s(E)-\mathcal P_s(M\times B_R)}
 {\mathcal P_s(M\times B_R)}
\ge C_1\bigl(k,s,R\sqrt{\lambda_1(M)}\bigr)
 \frac1{|E|}\int_M\frac{(|E_x|-A)^2}{|E_x|+A}\,d\mu(x),
\]
and
\[
\frac{\mathcal P_s(E)-\mathcal P_s(M\times B_R)}
 {\mathcal P_s(M\times B_R)}
\ge C_2(k,s)\alpha(E)^2.
\]
\end{theorem}

The coefficient $C_1$ is explicit, whereas $C_2$ also involves the constant in the Frank--Lieb theorem (Theorem~\ref{thm:FL}).
The thresholds are sufficient for each fixed exponent and depend on $s$.
We do not determine the optimal radius.

\subsection*{Periodic sets}
For a full-rank lattice $\Lambda\subset\mathbb R^m$ with fundamental cell $Q$, let the covering radius be $r_{\rm cov}(\Lambda)=\sup_{y\in\mathbb R^m}\operatorname{dist}(y,\Lambda)$. A set invariant under $\Lambda\times\{0\}$ has volume and perimeter measured in one cell; equivalently, it is a set on $(\mathbb R^m/\Lambda)\times\mathbb R^k$. Define its one-cell perimeter by
\begin{equation}\label{eq:periodic}
\mathcal P_s^\Lambda(\widetilde E)
=\int_{\widetilde E\cap(Q\times\mathbb R^k)}
 \int_{\mathbb R^{m+k}\setminus\widetilde E}
 \frac{dY\,dX}{|X-Y|^{m+k+s}}.
\end{equation}

The comparison sets $\mathbb R^m\times B_R$ are the straight cylinders of the periodic problem, with slabs when $k=1$.
\begin{corollary}\label{cor:torus}
Let $m,k\ge1$, $0<s<1$, and let $\Lambda\subset\mathbb R^m$ be
a full-rank lattice with fundamental cell $Q$ and covering radius
$r_{\rm cov}(\Lambda)$. Let $\widetilde E\subset\mathbb R^{m+k}$ be
measurable and invariant under translations by $\Lambda\times\{0\}$,
with $|\widetilde E\cap(Q\times\mathbb R^k)|=|Q|A$, $0<A<\infty$.
Put $R=(A/\omega_k)^{1/k}$. If
\[
R\ge\frac{T_{k,s}(1)}\pi r_{\rm cov}(\Lambda),
\]
then $\mathcal P_s^\Lambda(\widetilde E)
\ge\mathcal P_s^\Lambda(\mathbb R^m\times B_R)$, with equality
if and only if $\widetilde E$ agrees, up to a null set, with
$\mathbb R^m\times B_R(z_0)$ for some $z_0\in\mathbb R^k$.
For $m=1$ and $\Lambda=L\mathbb Z$, the condition is
$R\ge T_{k,s}(1)L/(2\pi)$.
\end{corollary}

By Lemma~\ref{lem:flat}, Theorem~\ref{thm:B} likewise applies to $\Lambda$-periodic
sets, with $M=\mathbb R^m/\Lambda$ and volumes measured in one cell. For a circle
of period $L$ the hypothesis reads $R\ge T_{k,s}(2)L/(2\pi)$, and
$R\sqrt{\lambda_1(M)}=2\pi R/L$.

\subsection*{Related work}
D\'avila, del Pino, Dipierro, and Valdinoci \cite{DDDV} constructed periodic minimizers in a cylindrically symmetric, even, monotone profile class and asked whether straight cylinders minimize at large volume. Cabr\'e, Csat\'o, and Mas proved existence and symmetry among all measurable periodic competitors and formulated the large-volume uniqueness conjecture \cite{CCM}. Corollary~\ref{cor:torus} proves this large-volume uniqueness conclusion in every ambient dimension $n\ge2$ for each fixed $s\in(0,1)$. For every $s_0\in(0,1)$, the results of \cite{CCM} also give $V_0>0$ such that, for $s\in(s_0,1)$, no singly periodic minimizer of volume less than $V_0$ per period is a cylinder. The second-variation criterion announced in \cite[Theorem~1.3]{CCM} concerns nonnegative second variation under volume-preserving periodic variations and is distinct from the global large-volume comparison proved here.

Alcover and Bruera \cite{AB} proved existence and symmetry of periodic minimizers
for a class of anisotropic kernels in every dimension $n\ge2$ and constructed planar
bifurcating families with constant anisotropic nonlocal mean curvature. Bruera
\cite{Bruera} classified the second-variation stability of planar strips in the
even, transversely symmetric periodic class and proved, in certain parameter
ranges, rigidity of nearby stable critical points. In the planar singly periodic
case the author proved in \cite{Seo} that strips of half-width $R\ge14L$, where
$L$ is the period, are the unique minimizers for every $s\in(0,1)$, with a
remainder controlling the slice lengths. The present paper
uses the same gain-versus-loss mechanism for arbitrary closed bases and
every $k\ge1$, and adds control of the asymmetry $\alpha(E)$; the price is a
threshold that depends on $s$.

For the Euclidean fractional perimeter, balls minimize at every volume \cite{FS},
with quantitative refinements in \cite{FFMMM, FMM}. For a circle base, cylindrical rearrangement with common-center equality
and the circle heat-kernel representation are treated in \cite{CCM};
\cite{CM} establishes rearrangement inequalities and equality cases
for periodic Gagliardo seminorms under its hypotheses. Here only the
Euclidean slices are rearranged; the geometry of the base enters through
its spectral gap.

\subsection*{Strategy of the proof}
Rearranging each slice into a ball does not suffice by itself: the energy of a ball is concave in its volume, so averaging the ball energies over the base favors unequal slice volumes. The proof shows that the interaction between slices more than compensates for this loss when the cylinder is sufficiently wide.

Section~\ref{sec:heat} develops the heat-kernel formulation and the reduction to slice volumes.
Let $\bar\rho_E(z)=|M|^{-1}\mu(\{x\in M:(x,z)\in E\})$ be the marginal of $E$ on
$\mathbb R^k$. It has mass $A$, and it measures the distance from $E$ itself to
cylinders:
$|E\triangle(M\times B)|=|M|\,\|\bar\rho_E-\chi_B\|_1$ for every
measurable $B\subset\mathbb R^k$.

Expanding in eigenfunctions of $M$ bounds the heat that has left $E$ at time $t$
from below by two terms: the heat lost in $\mathbb R^k$ by the density
$\bar\rho_E$, with weight $1-e^{-\lambda_1(M)t}$, and the average over $x\in M$
of the heat lost by the slices $E_x$, with weight $e^{-\lambda_1(M)t}$
(Lemma~\ref{lem:spectral}). Rearranging each slice into a centered ball and
integrating in $t$ gives (Proposition~\ref{prop:reduction})
\[
\begin{aligned}
\mathcal P_s(E)-\mathcal P_s(M\times B_R)
 &\ge \frac{\kappa_{n,s}}{\kappa_{k,s}}|M|
 \left(D(\bar\rho_E)-\frac1{|M|}\int_M\mathcal L(|E_x|)\,d\mu(x)\right),\\
D(\bar\rho_E)&\ge\frac1{|M|}\int_M\mathcal G(|E_x|)\,d\mu(x).
\end{aligned}
\]
Here $D(\bar\rho_E)\ge0$ is the amount by which the self-interaction of
$\bar\rho_E$ falls short of that of $\chi_{B_R}$, for an integrable radial kernel
$K_\lambda$ built from the weight $1-e^{-\lambda_1(M)t}$. The functions $\mathcal G$ and
$\mathcal L$ are nonnegative and vanish at $A$: the \emph{gain} $\mathcal G$ is a
lower bound for this deficit forced by unequal slice volumes, and the \emph{loss}
$\mathcal L$ comes from the concavity, as a function of volume, of the heat loss
of a ball, weighted by $e^{-\lambda_1(M)t}$.

The core of the paper is the one-variable comparison
(Proposition~\ref{prop:scalar})
\[
\mathcal G(a)\ge\theta\,\mathcal L(a)\quad(a\ge0),
\]
whenever $R\sqrt{\lambda_1(M)}\ge T_{k,s}(\theta)$ and $\theta\in\{1,2\}$.
Up to sign, the derivatives of $\mathcal G$ and $\mathcal L$ are $V-V(R)$ and
$h-h(R)$, where $V$ is the potential of $B_R$ for the kernel $K_\lambda$ and $h$
is the nonlocal mean curvature of spheres for the complementary kernel
$|z|^{-k-s}-K_\lambda(|z|)$. The comparison follows from convexity properties of
$V$, obtained by writing $K_\lambda$ as a positive superposition of resolvent
kernels, and from bounds on $h$ and $h'$ (Section~\ref{sec:scalar}). With
$\theta=1$ the gain absorbs the loss, which proves the inequality of
Theorem~\ref{thm:A}; equality forces equal slice volumes and then, by the
equality cases of the rearrangement inequalities, a single translated cylinder
(subsection~\ref{sec:proofA}). With $\theta=2$ half of $D(\bar\rho_E)$ survives. A
lower bound for $\mathcal G$ then gives the first estimate of
Theorem~\ref{thm:B}, and the quantitative Riesz inequality of Frank and Lieb,
applied to $K_\lambda$ through a layer-cake decomposition, gives the second
(subsection~\ref{sec:proofB}). Finally, subsection~\ref{sec:periodic} transfers the conclusions to flat tori.

\section{Heat-kernel formulation and reduction}\label{sec:heat}\label{sec:reduction}
Throughout, $k$, $s$, $M$ and $A>0$ are fixed, $R=(A/\omega_k)^{1/k}$, and
$\lambda=\lambda_1(M)$. Balls $B_r(z_0)\subset\mathbb R^k$ are open,
$B_r=B_r(0)$, and $e_1$ is the first coordinate vector. We write
$\fint_Mf\,d\mu=|M|^{-1}\int_Mf\,d\mu$, $\|\cdot\|_1$ for the norm of
$L^1(\mathbb R^k)$, and, for a density $\rho$ on $\mathbb R^k$,
$\iint\rho\,G_t\,\rho=\iint\rho(z)G_t(z-w)\rho(w)\,dz\,dw$.
For an integer $d\ge1$, $G_t^{(d)}(z)=(4\pi t)^{-d/2}e^{-|z|^2/(4t)}$ denotes the Gaussian on $\mathbb R^d$, so that $G_t=G_t^{(k)}$, and $B_1^d$ denotes the open unit ball of $\mathbb R^d$.
\subsection{Heat kernels and normalization}
The heat-kernel facts below can be found in \cite{Chavel, Grigoryan}. The heat kernel $h_t(x,y)$ of $M$ is smooth on $(0,\infty)\times M\times M$ and symmetric in $(x,y)$, and it is strictly positive. It satisfies
\begin{equation}\label{eq:mass}
 \int_Mh_t(x,y)\dd\mu(y)=1\qquad(x\in M,\ t>0).
\end{equation}
Let $0=\lambda_0<\lambda_1\le\lambda_2\le\cdots$ be the eigenvalues of $-\Delta_M$, repeated according to multiplicity. For these eigenvalues, let $\{\varphi_j\}$ be an orthonormal basis of $L^2(M)$ consisting of associated real eigenfunctions, with $\varphi_0=|M|^{-1/2}$. Then
\begin{equation}\label{eq:base-spectral-expansion}
 h_t(x,y)=\sum_{j\ge0}e^{-\lambda_jt}\varphi_j(x)\varphi_j(y),
\end{equation}
with absolute and uniform convergence on $[t_0,\infty)\times M\times M$ for each $t_0>0$.

For $E\subset N$ of finite measure and $t>0$ let
\begin{equation}\label{eq:Ht}
H_t(E)=\int_E\int_{N\setminus E}h_t^N(X,Y)\,dY\,dX
=|E|-\int_E\int_Eh_t^N(X,Y)\,dY\,dX
\end{equation}
be the heat that has left $E$ at time $t$; the second expression uses
$\int_Nh_t^N(X,Y)\,dY=1$, which follows from \eqref{eq:mass}. Thus
$\mathcal P_s(E)=\kappa_{n,s}\int_0^\infty H_t(E)\,t^{-1-s/2}\,dt$.
With the normalization in \eqref{eq:def} and the kernel convention in
\cite[Definition~2.3, equation~(9)]{CFS2}, $\mathcal P_s(E)$ is
$\kappa_{n,s}\Gamma(1-s/2)/s$ times the squared heat-kernel
fractional seminorm of $\chi_E$ on $N$, allowing the value $+\infty$.
Equivalently, up to normalization, it is the quadratic form of
$(-\Delta_N)^{s/2}$ at $\chi_E$.

For a measurable set $F\subset\R^k$ of finite measure let
\[
 P^{(k)}_s(F)=\int_F\int_{\R^k\setminus F}\frac{\dd w\dd z}{|z-w|^{k+s}}\in[0,+\infty]
\]
be its fractional perimeter. The following identity fixes the heat-kernel normalization relative to the Euclidean kernel.

\begin{lemma}\label{lem:normalization}
\begin{enumerate}
\item[(a)] Let $k\ge1$ and $0<s<1$. For every $r>0$,
\begin{equation}\label{eq:subordination}
 \kappa_{k,s}\int_0^\infty(4\pi t)^{-k/2}e^{-r^2/(4t)}\,\frac{\dd t}{t^{1+s/2}}=r^{-k-s}.
\end{equation}
Consequently, for every measurable $F\subset\R^k$ with $|F|<\infty$,
\begin{equation}\label{eq:euclideanheat}
 P^{(k)}_s(F)=\kappa_{k,s}\int_0^\infty\Big(|F|-\int_F\int_FG_t(z-w)\dd w\dd z\Big)\frac{\dd t}{t^{1+s/2}} .
\end{equation}
\item[(b)] Let $F\subset\R^k$ be measurable with $|F|<\infty$. Then
\[
\cP_s(M\times F)=c_{m,k,s}\,|M|\,P^{(k)}_s(F),\qquad c_{m,k,s}=\frac{\kappa_{n,s}}{\kappa_{k,s}}=\frac{\pi^{m/2}\,\Gamma\big(\frac{k+s}2\big)}{\Gamma\big(\frac{n+s}2\big)} .
\]
\item[(c)] For $k\ge1$ and $0<s<1$ the fractional perimeter of the unit ball of $\R^k$ equals $b_{k,s}$, where
\begin{equation}\label{eq:constants1}
b_{k,s}=P_s^{(k)}(B_1^k)=\frac{2^{1-s}\pi^{(k-1)/2}k\omega_k\Gamma((1-s)/2)}{s(k-s)\Gamma((k-s)/2)}.
\end{equation}
 Consequently $P_s^{(k)}(B_R)=b_{k,s} R^{k-s}$ and
\begin{equation}\label{eq:cylinder}
 \cP_s\big(M\times B_R\big)=c_{m,k,s}\,b_{k,s}\,|M|\,R^{k-s}\qquad(R>0).
\end{equation}
\end{enumerate}
\end{lemma}
\begin{proof}
(a) Substituting $\sigma=r^2/(4t)$ rewrites the integral in \eqref{eq:subordination} as
\[
 (4\pi)^{-k/2}\Big(\frac{r^2}4\Big)^{-(k+s)/2}\int_0^\infty \sigma^{(k+s)/2-1}e^{-\sigma}\dd\sigma
 =2^s\pi^{-k/2}\,\Gamma\Big(\frac{k+s}2\Big)\,r^{-k-s}.
\]
For \eqref{eq:euclideanheat}, insert \eqref{eq:subordination} into the definition of $P_s^{(k)}(F)$ and use Gaussian mass conservation:
\[
\int_F\int_{\R^k\setminus F}G_t(z-w)\dd w\dd z
=|F|-\int_F\int_FG_t(z-w)\dd w\dd z.
\]

(b) For $E=M\times F$, \eqref{eq:Ht} and \eqref{eq:mass} give
\begin{align*}
 H_t(M\times F)&=|M|\,|F|-\int_M\int_Mh_t(x,y)\dd\mu(y)\dd\mu(x)\int_F\int_FG_t(z-w)\dd w\dd z\\
 &=|M|\Big(|F|-\int_F\int_FG_t(z-w)\dd w\dd z\Big).
\end{align*}
Multiply by $\kappa_{n,s}t^{-1-s/2}$, integrate in $t$, and apply \eqref{eq:euclideanheat}.

(c) Formula \eqref{eq:constants1} is proved in Appendix~\ref{app:ball}. The scaling $P_s^{(k)}(B_R)=R^{k-s}P_s^{(k)}(B_1^k)$ follows from the change of variables $z\mapsto Rz$, $w\mapsto Rw$, and \eqref{eq:cylinder} from Lemma~\ref{lem:normalization}(b).
\end{proof}

\subsection{Slices, marginal and spectral splitting}
Let $E\subset N$ be measurable with $|E|<\infty$. The slice $E_x$ is measurable for almost every $x\in M$, and the slice-volume function $a_E(x)=|E_x|$ is a measurable function with $\int_Ma_E\dd\mu=|E|$. Set $a_E(x)=0$ on the null set.

Define the marginal of $E$ by
\[
\bar\rho_E(z)=\frac{\mu(\{x\in M:(x,z)\in E\})}{|M|}\in[0,1].
\]
Then we have $\int_{\R^k}\bar\rho_E=A$ and
\begin{equation}\label{eq:alpharho}
|E\triangle(M\times B)|=|M|\,\|\bar\rho_E-\chi_B\|_1
\quad\text{for every measurable }B\subset\R^k.
\end{equation}
Thus $\alpha(E)=A^{-1}\inf_{z_0}\|\bar\rho_E-\chi_{B_R(z_0)}\|_1$.

\begin{lemma}\label{lem:spectral}
Let $E\subset N$ be measurable with $|E|=|M|A<\infty$. For every $t>0$,
\begin{equation}\label{eq:spectral}
\begin{aligned}
H_t(E)&\ge\big(1-e^{-\lambda t}\big)|M|\Big(A-\iint\bar\rho_E\,G_t\,\bar\rho_E\Big)\\
&+e^{-\lambda t}\int_M\Big(|E_x|-\iint_{E_x\times E_x}G_t(z-w)\dd z\dd w\Big)\dd\mu(x).
\end{aligned}
\end{equation}
\end{lemma}
\begin{proof}
Since $|E|<\infty$, $\chi_E\in L^2(N)$.
Put $g(x,\cdot)=G_{t/2}*\chi_{E_x}$. By the semigroup property and the product structure $h^N_t=h_t\,G_t$,
\[
\int_E\int_E h^N_t=\big\|e^{\frac t2\Delta_N}\chi_E\big\|^2_{L^2(N)}=\int_{\R^k}\big\|e^{\frac t2\Delta_M}g(\cdot,z)\big\|^2_{L^2(M)}\dd z .
\]
For $\varphi\in L^2(M)$ with mean $\bar\varphi$ the spectral expansion \eqref{eq:base-spectral-expansion} gives
\[
\|e^{\frac t2\Delta_M}\varphi\|^2\le|M|\bar\varphi^{\,2}+e^{-\lambda t}(\|\varphi\|^2-|M|\bar\varphi^{\,2}).
\]
 Apply this with $\varphi=g(\cdot,z)$, whose mean is $(G_{t/2}*\bar\rho_E)(z)$, and integrate in $z$:
\[
\int_E\int_E h^N_t\le\big(1-e^{-\lambda t}\big)|M|\,\|G_{t/2}*\bar\rho_E\|^2_{L^2(\R^k)}+e^{-\lambda t}\|g\|^2_{L^2(N)} .
\]
Since $\|G_{t/2}*\bar\rho_E\|^2=\iint\bar\rho_EG_t\bar\rho_E$ and $\|g\|^2=\int_M\iint_{E_x\times E_x}G_t\dd\mu$, the claim follows from \eqref{eq:Ht} and $|E|=|M|A=\int_M|E_x|\dd\mu$.
\end{proof}

The two spectral factors induce the positive measures
\[
\dd\nu_{\rm gap}(t)=\kappa_{k,s}(1-e^{-\lambda t})t^{-1-s/2}\dd t,\qquad \dd\nu_{\rm rem}(t)=\kappa_{k,s}e^{-\lambda t}t^{-1-s/2}\dd t.
\]

Their sum is $\kappa_{k,s}t^{-1-s/2}\dd t$. An integration by parts gives
\begin{equation}\label{eq:gapmass}
\ell:=\int_0^\infty\dd\nu_{\rm gap}=\frac{2\kappa_{k,s}\Gamma(1-s/2)}s\lambda^{s/2}.
\end{equation}

We also record the following moment of the remainder measure:
\begin{equation}\label{eq:eta}
\eta:=\int_0^\infty\sqrt{t/\pi}\,d\nu_{\rm rem}(t)
=\frac{\kappa_{k,s}\Gamma((1-s)/2)}{\sqrt\pi}\,\lambda^{(s-1)/2}.
\end{equation}

The gap measure has finite mass $\ell$; the remainder measure has infinite total mass. Weighted integrals involving ball heat loss will be justified separately.
\subsection{Ball functions and radial comparison}
For $p\ge0$ put $r(p)=(p/\omega_k)^{1/k}$ and let $B_{r(p)}\subset\R^k$ be the centered open ball of volume $p$, with $r(0)=0$ and $B_{r(0)}=\varnothing$. For $p,q\ge0$ and $t>0$ denote their Gaussian interaction by
\[
J_t(p,q)=\int_{B_{r(p)}}\int_{B_{r(q)}}G_t(z-w)\dd w\dd z .
\]

Define $U_t(p)=p-J_t(p,p)$. Gaussian mass conservation gives $0\le J_t(p,q)\le\min\{p,q\}$ and $0\le U_t(p)\le p$.

For $t>0$, define $u_t(r)$ for $r\ge0$ and $\delta_t(r)$ for $r>0$ by
\[
u_t(r)=\int_{B_R}G_t(re_1-w)\dd w,\qquad \delta_t(r)=\frac12-\int_{B_r}G_t(re_1-w)\dd w.
\]
Thus $u_t$ is the heat potential of the fixed ball $B_R$ evaluated
at distance $r$ from its center, whereas $\delta_t(r)$ is the amount
by which the heat potential of the variable ball $B_r$ falls short
of $1/2$ on its boundary.
The following properties give the concavity needed for the tangent gaps.
\begin{lemma}\label{lem:ballfunctions}
For $p>0$ the volume derivatives satisfy
\begin{equation}\label{eq:dJ}
\partial_pJ_t(p,A)=u_t(r(p)),\qquad \frac{d}{dp}J_t(p,p)=1-2\delta_t(r(p)),\qquad U_t'(p)=2\delta_t(r(p)).
\end{equation}

The function $u_t$ is nonincreasing on $[0,\infty)$, and $\delta_t$ is decreasing on $(0,\infty)$, with
\begin{equation}\label{eq:deltamon}
0<\delta_t(r)<\frac12\qquad(r>0),\qquad \delta_t(r)\uparrow\frac12\ (r\downarrow0),\qquad \delta_t(r)\downarrow0\ (r\uparrow\infty).
\end{equation}

Both $J_t(\cdot,A)$ and $U_t$ are concave. Moreover,
\begin{equation}\label{eq:heatloss}
U_t(p)\le\sqrt{t/\pi}\,k\omega_k r(p)^{k-1},\qquad \delta_t(r)\le\frac{k}{2\sqrt\pi}\frac{\sqrt t}{r}.
\end{equation}
\end{lemma}
\begin{proof}
The convolution $G_t*\chi_{B_R}$ is radial and radially nonincreasing; differentiation with respect to volume gives \eqref{eq:dJ}.
Moreover $\int_{B_r}G_t(re_1-w)\,dw=\int_{B_r(-re_1)}G_t$, and the balls
$B_r(-re_1)=\{\zeta:|\zeta|^2<-2r\,\zeta\cdot e_1\}$ increase with $r$, shrink to
a null set as $r\downarrow0$, and exhaust the half-space $\{\zeta\cdot e_1<0\}$,
of Gaussian mass $\frac12$, as $r\uparrow\infty$. Since $G_t>0$, this gives the
monotonicity of $\delta_t$ and \eqref{eq:deltamon}; concavity follows from
\eqref{eq:dJ}. For the heat-loss bound, since
$|B_{r(p)}\setminus(B_{r(p)}-\zeta)|\le\omega_{k-1}r(p)^{k-1}|\zeta|$,
\[
\begin{aligned}
U_t(p)=\int_{\mathbb R^k}G_t(\zeta)\,|B_{r(p)}\setminus(B_{r(p)}-\zeta)|\,d\zeta
&\le\omega_{k-1}r(p)^{k-1}\int_{\mathbb R^k}G_t(\zeta)|\zeta|\,d\zeta\\
&=\omega_{k-1}r(p)^{k-1}\,\frac{2\sqrt t\,\Gamma((k+1)/2)}{\Gamma(k/2)}\\
&=\sqrt{t/\pi}\,k\omega_k\,r(p)^{k-1}.
\end{aligned}
\]
Since $U_t$ is concave and vanishes at $0$, we have $pU_t'(p)\le U_t(p)$, which gives the second bound in \eqref{eq:heatloss}.
\end{proof}

By Lemma~\ref{lem:normalization}, the integrated ball identity is
\begin{equation}\label{eq:P}
\kappa_{k,s}\int_0^\infty U_t(p)\frac{\dd t}{t^{1+s/2}}=b_{k,s}r(p)^{k-s}\qquad(p\ge0).
\end{equation}

At $p=0$ the right side is zero, since $k-s>0$.
We use the Riesz rearrangement inequality \cite{Riesz} in the following two forms, both covered by \cite[Theorem~2]{BH}. If
$F_1,F_2\subset\mathbb R^k$ are measurable sets of finite measure and $t>0$, then
\begin{equation}\label{eq:riesz}
\int_{F_1}\int_{F_2}G_t(z-w)\,dw\,dz\le J_t(|F_1|,|F_2|).
\end{equation}
If $\rho:\mathbb R^k\to[0,1]$ is integrable with symmetric decreasing
rearrangement $\rho^*$, and $K\ge0$ is radial and radially nonincreasing, then
\begin{equation}\label{eq:rieszfn}
\iint\rho(z)K(z-w)\rho(w)\,dz\,dw\le\iint\rho^*(z)K(z-w)\rho^*(w)\,dz\,dw.
\end{equation}

The next lemma compares a radial density with the indicator of the
ball of the same mass. In Proposition~\ref{prop:reduction} it is applied
to an average of indicators of centered balls, which in general is not
the symmetric decreasing rearrangement of $\bar\rho_E$.

\begin{lemma}\label{lem:radialcomparison}
Let $K$ be a nonnegative bounded integrable function on $\mathbb R^k$
that is radial and radially nonincreasing. Let $f:\mathbb R^k\to[0,1]$
be radial and radially nonincreasing with $\int f=|B_R|$, where $R>0$.
Write $\chi=\chi_{B_R}$. Then
\[
\iint[\chi(z)\chi(w)-f(z)f(w)]K(z-w)\,dz\,dw
\ge\int(\chi-f)(K*\chi)\,dz\ge0.
\]
If $K*\chi$ is strictly radially decreasing, equality of the two
quadratic interactions implies $f=\chi$ almost everywhere.
\end{lemma}
\begin{proof}
Set $g=\chi-f$. Then $g\ge0$ inside $B_R$, $g\le0$ outside, and $\int g=0$.
The convolution $K*f$ is radial and radially nonincreasing, so
\[
\int g(K*f)=\int g[(K*f)-(K*f)(Re_1)]\ge0.
\]
The same holds for $K*\chi$. By the symmetry of $K$, the difference of the two quadratic interactions equals $\int g[K*(\chi+f)]$, which proves both inequalities.
If $K*\chi$ is strictly radially decreasing, the integrand $g[(K*\chi)-(K*\chi)(Re_1)]$ is positive wherever $g\ne0$, except on the null set $\partial B_R$. Hence $\int g(K*\chi)>0$ unless $f=\chi$ almost everywhere, and equality of the two quadratic interactions forces $f=\chi$ almost everywhere.
Boundedness, integrability, $0\le f\le1$ and $f\in L^1$ justify every integral.
\end{proof}

\subsection{Tangent gaps and the central reduction}

For $a\ge0$ use the convention $\int_a^A=-\int_A^a$ and define
\begin{equation}\label{eq:Wt}
\mathcal G_t(a)=\int_a^A[u_t(r(v))-u_t(R)]\dd v
=J_t(A,A)-J_t(a,A)-u_t(R)(A-a),
\end{equation}
\begin{equation}\label{eq:ft}
\mathcal L_t(a)=2\int_a^A[\delta_t(r(v))-\delta_t(R)]\dd v
=U_t(A)+U_t'(A)(a-A)-U_t(a).
\end{equation}

Thus $\mathcal G_t(a)$ and $\mathcal L_t(a)$ are the amounts by which the concave
functions $J_t(\cdot,A)$ and $U_t$ lie below their tangent lines at $A$. Both are
nonnegative, and $\mathcal G_t(a)\le|A-a|$.
Their integrated forms are
\begin{equation}\label{eq:WF}
\mathcal G(a)=\int_0^\infty\mathcal G_t(a)\dd\nu_{\rm gap}(t),\qquad \mathcal L(a)=\int_0^\infty\mathcal L_t(a)\dd\nu_{\rm rem}(t).
\end{equation}

The gain $\mathcal G$ is bounded by the interaction deficit in Proposition~\ref{prop:reduction}, while $\mathcal L$ is the concavity loss. Both are finite:
\[
\mathcal G(a)\le\ell|A-a|,\qquad \mathcal L(a)\le b_{k,s}R^{k-s} +|a-A|\int_0^\infty U_t'(A)\dd\nu_{\rm rem}(t)<\infty.
\]

Indeed $\mathcal L_t(a)\le U_t(A)+U_t'(A)|a-A|$; \eqref{eq:P} bounds the first integral by $b_{k,s}R^{k-s}$, and \eqref{eq:dJ}--\eqref{eq:heatloss} bound the second by a multiple of the integrable function $e^{-\lambda t}t^{-(1+s)/2}$.

We now define the deficit that the reduction retains. For measurable
$\rho:\mathbb R^k\to[0,1]$ with $\int\rho=A$ let
\begin{equation}\label{eq:D}
D_t(\rho)=J_t(A,A)-\iint\rho\,G_t\,\rho,\qquad
D(\rho)=\int_0^\infty D_t(\rho)\,d\nu_{\rm gap}(t).
\end{equation}
Then,
\begin{equation}\label{eq:DK}
\begin{aligned}
D(\rho)&=\iint\bigl[\chi_{B_R}(z)\chi_{B_R}(w)-\rho(z)\rho(w)\bigr]
K_\lambda(|z-w|)\,dz\,dw,\\
K_\lambda(r)&=\int_0^\infty G_t(re_1)\,d\nu_{\rm gap}(t).
\end{aligned}
\end{equation}
The kernel $K_\lambda$ is integrable, with $\int_{\mathbb R^k}K_\lambda(|z|)\,dz
=\ell$; in particular $|D(\rho)|\le\ell A$.
\begin{proposition}\label{prop:reduction}
Let $E\subset N$ be measurable with $|E|=|M|A$, and let $a_E(x)=|E_x|$. Then
\[
\mathcal P_s(E)-\mathcal P_s(M\times B_R)\ge c_{m,k,s}|M|\left(D(\bar\rho_E)-\fint_M\mathcal L(a_E)\dd\mu\right).
\]

Moreover, for every $t>0$,
\[
D_t(\bar\rho_E)\ge\fint_M\mathcal G_t(a_E)\dd\mu\ge0,\quad
D(\bar\rho_E)\ge\fint_M\mathcal G(a_E)\dd\mu.
\]
\end{proposition}
\begin{proof}
We first prove the bounds for $D_t$ and $D$. Let
\[
f(z):=\fint_M\chi_{B_{r(a_E(x))}}(z)\,\dd\mu(x).
\]
Then $f$ is measurable, radial and radially nonincreasing, $0\le f\le1$, and
$\int_{\mathbb R^k}f\,\dd z=A$. By the definition of $\bar\rho_E$ and the Riesz inequality \eqref{eq:riesz},
\begin{align*}
\iint\bar\rho_EG_t\bar\rho_E
&=\fint_M\fint_M\iint_{E_x\times E_y}G_t(z-w)\dd z\dd w\dd\mu(x)\dd\mu(y)\\
&\le\fint_M\fint_MJ_t(a_E(x),a_E(y))\dd\mu(x)\dd\mu(y)\\
&=\iint fG_tf.
\end{align*}
Lemma~\ref{lem:radialcomparison} with $K=G_t$ and this $f$ gives, with $\chi=\chi_{B_R}$,
\[
\begin{aligned}
J_t(A,A)-\iint fG_tf&\ge\int(\chi-f)(G_t*\chi)\dd z\\
&=\fint_M[J_t(A,A)-J_t(a_E(x),A)]\dd\mu(x)\\
&=\fint_M\mathcal G_t(a_E)\dd\mu,
\end{aligned}
\]
where the last identity uses \eqref{eq:Wt} and $\fint_M(A-a_E)\dd\mu=0$. Integrating in $t$ gives the bound for $D$.

For the perimeter bound we may assume $\cP_s(E)<\infty$. By \eqref{eq:riesz} with $F_1=F_2=E_x$, the last bracket in \eqref{eq:spectral} is at least $U_t(a_E(x))$, and by \eqref{eq:ft} and $\fint_M(a_E-A)\dd\mu=0$ we have 
\[
\fint_MU_t(a_E)\dd\mu=U_t(A)-\fint_M\mathcal L_t(a_E)\dd\mu
.
\]
 Hence \eqref{eq:spectral} gives
\[
\frac{H_t(E)}{|M|}\ge U_t(A)+\big(1-e^{-\lambda t}\big)D_t(\bar\rho_E)-e^{-\lambda t}\fint_M\mathcal L_t(a_E)\dd\mu .
\]
Multiply by $\kappa_{n,s}t^{-1-s/2}=c_{m,k,s}\kappa_{k,s}\,t^{-1-s/2}$ and integrate. By \eqref{eq:P}, multiplied by $c_{m,k,s}$, and \eqref{eq:cylinder} the first term gives $\cP_s(M\times B_R)/|M|$. The other two terms have constant signs, and 
\[
\int_0^\infty\!\fint_M\mathcal L_t(a_E)\dd\mu\dd\nu_{\rm rem}(t)=\fint_M\mathcal L(a_E)\dd\mu
\le b_{k,s}R^{k-s}<\infty
\]
 because $\fint_M \mathcal L_t(a_E)\dd\mu=U_t(A)-\fint_M U_t(a_E)\dd\mu\le U_t(A)$. This proves the perimeter bound.
\end{proof}

Proposition~\ref{prop:reduction} reduces the inequality of Theorem~\ref{thm:A}
to a statement about functions of one variable. Indeed, since
$D(\bar\rho_E)\ge\fint_M\mathcal G(a_E)\dd\mu$, the perimeter deficit of $E$ is at
least $c_{m,k,s}\int_M[\mathcal G(a_E)-\mathcal L(a_E)]\dd\mu$, which is
nonnegative as soon as $\mathcal G\ge\mathcal L$ on $[0,\infty)$.
Section~\ref{sec:scalar} proves this inequality, and the stronger inequality
$\mathcal G\ge2\mathcal L$ needed for Theorem~\ref{thm:B}, under the
corresponding lower bounds on $R\sqrt\lambda$.

\section{The one-variable gain-versus-loss comparison}\label{sec:scalar}
The gain-versus-loss comparison of Proposition~\ref{prop:scalar} absorbs the concavity loss at every nonnegative slice volume.
\subsection{Integral representations and thresholds}
Define the integrated ball potential $V(r)$ for $r\ge0$ using the gap measure and the curvature $h(r)$ for $r>0$ using the remainder measure by
\[
V(r)=\int_0^\infty u_t(r)\dd\nu_{\rm gap}(t),\qquad h(r)=2\int_0^\infty\delta_t(r)\dd\nu_{\rm rem}(t).
\]
$V$ is the potential of the ball $B_R$ for the kernel $K_\lambda$, and $h$ is the nonlocal mean curvature of $\partial B_r$ for the complementary kernel $|z|^{-k-s}-K_\lambda(|z|)$. The ball energy associated with the remainder measure is the heat loss $\int_0^\infty U_t(A)\dd\nu_{\rm rem}(t)$.

Insert \eqref{eq:Wt}, \eqref{eq:ft} into \eqref{eq:WF}; the integrands have a fixed sign on each side of $A$. For every $a\ge0$,
\begin{equation}\label{eq:WFint}
\mathcal G(a)=\int_a^A[V(r(v))-V(R)]\dd v,\qquad \mathcal L(a)=\int_a^A[h(r(v))-h(R)]\dd v.
\end{equation}
\begin{definition}\label{def:rho}
Let $\theta\in\{1,2\}$. For $x>0$ consider condition $(\mathrm I_\theta)$:
\begin{equation*}\tag{$\mathrm I_\theta$}
\frac{2\kappa_{k,s}\Gamma(1-s/2)}s x^s
\ge\left(1+\frac sk\right)\frac{b_{k,s}}{\omega_k}
+\frac{2(\theta-1)k\kappa_{k,s}\Gamma((1-s)/2)}{\sqrt\pi}
 x^{s-1}.
\end{equation*}

The dimension-dependent condition $(\mathrm{II}_\theta)$ is
\begin{align*}
k=1:&\quad x\ge\ln(1+2\sqrt\theta);
\\
k=2:&\quad
\frac{2\kappa_{2,s}\Gamma(1-s/2)}s x^s
\ge(1-s/2)\frac{b_{2,s}}{\omega_2}+\omega_2\ln3
+\frac{4\theta\kappa_{2,s}\Gamma((1-s)/2)}{\sqrt\pi}x^{s-1};
\\
k\ge3:&\quad
\frac{2\kappa_{k,s}\Gamma(1-s/2)}s x^s
\ge\left(1-\frac sk\right)\frac{b_{k,s}}{\omega_k}
+\frac{(2\theta-1)(k-1)\kappa_{k,s}\Gamma((1-s)/2)}{\sqrt\pi}x^{s-1}.
\end{align*}
Let $T_{k,s}(\theta)$ be the smallest $x_0>0$ for which both
conditions hold for all $x\ge x_0$.
\end{definition}
In each condition, the difference of the two sides is continuous and
strictly increasing in $x$, negative as $x\downarrow0$, and tends to
$+\infty$ as $x\to\infty$. Hence each condition holds exactly on a
closed half-line, and $T_{k,s}(\theta)$ is well defined, finite and positive.
Since the right-hand sides are nondecreasing in $\theta$,
$T_{k,s}(2)\ge T_{k,s}(1)>0$.
\begin{proposition}\label{prop:scalar}
Let $\theta\in\{1,2\}$ and $R\sqrt{\lambda}\ge T_{k,s}(\theta)$. Then $\mathcal G\ge\theta \mathcal L$ on $[0,\infty)$, and $\mathcal G-\theta \mathcal L>0$ on $(0,A)\cup(A,\infty)$. If $\theta=1$, then also $(\mathcal G-\mathcal L)(0)>0$.
\end{proposition}

The proposition is proved at the end of this section. Lemma~\ref{lem:V} gives
the convexity properties of $V$ and Lemma~\ref{lem:h} the bounds on $h$;
Lemma~\ref{lem:formulas} links the two at $r=R$. The range $a<A$ then reduces to
the sign of $\mathcal G-\theta\mathcal L$ at $a=0$ (Lemma~\ref{lem:inside}),
which is condition $(\mathrm I_\theta)$. The range $a>A$ is treated separately
for $k=1$, $k=2$ and $k\ge3$ (Lemma~\ref{lem:outside}), which gives the three
forms of $(\mathrm{II}_\theta)$.
\subsection{Ball potential and cone bound}
\begin{lemma}\label{lem:V}
The following properties hold.
\begin{enumerate}
\item[(i)] The function $V$ satisfies $0\le V\le\ell$, is nonincreasing on $[0,\infty)$, and tends to zero at infinity.
\item[(ii)] The function $V$ is of class $C^1$ on $(0,\infty)$; $-V'\ge0$ is nondecreasing on $(0,R]$ and $r\mapsto r^{k-1}(-V'(r))$ is nonincreasing on $[R,\infty)$. In particular, $V$ is concave on $[0,R]$ and convex on $[R,\infty)$.
\item[(iii)] If $k=1$, then $-V'(r)=\int_0^\infty [G^{(1)}_t(r-R)-G^{(1)}_t(r+R)]\dd\nu_{\rm gap}(t)$.
\item[(iv)] Moreover,
\begin{equation}\label{eq:cone}
V(2R)\le k2^{-k}\omega_k\ln3\,R^{-s}.
\end{equation}
\end{enumerate}
\end{lemma}
\begin{proof}
$0\le u_t\le1$ and \eqref{eq:gapmass} give $0\le V\le\ell$; monotonicity is inherited from $u_t$, and 
\[
V(r)\le\int_{B_R}|r e_1-w|^{-k-s}\dd w\to0
\]
 because $K_\lambda(|z|)\le|z|^{-k-s}$ by \eqref{eq:subordination}.

Two finite Gamma integrals yield the positive Laplace-transform representation
\[
(1-e^{-\lambda t})t^{-1-s/2}=\int_0^\infty e^{-qt}m_\lambda(q)\dd q,\qquad m_\lambda(q)=\frac{q^{s/2}-(q-\lambda)_+^{s/2}}{\Gamma(1+s/2)}>0.
\]

Let $v_q=\int_0^\infty e^{-qt}(G_t*\chi_{B_R})\dd t$. The resolvent identity, positivity, and Tonelli's theorem give
\[
(-\Delta+q)v_q=\chi_{B_R},\qquad 0<v_q<q^{-1},\qquad
V=\kappa_{k,s}\int_0^\infty m_\lambda(q)v_q\dd q.
\]

The radial potential is nonincreasing. The identity $\|\partial_{e_1}G_t\|_1=(\pi t)^{-1/2}$ gives $\|\partial_{e_1}v_q\|_\infty\le q^{-1/2}$. As $q\downarrow0$, $m_\lambda(q)=O(q^{s/2})$; as $q\to\infty$, $m_\lambda(q)=O(q^{s/2-1})$. Thus both $m_\lambda/q$ and $m_\lambda/\sqrt q$ are integrable, the latter because $s/2<1/2$. Absolute domination justifies first differentiation under the integral and $C^1$ matching at $R$.
Inside $B_R$ the function $q^{-1}-v_q$ is radial, smooth and positive, and it
solves $\Delta w=qw$. The radial solutions of this equation that are smooth at
the origin (even, if $k=1$) are the multiples of
\[
\phi_q(r)=\frac1{|S^{k-1}|}\int_{S^{k-1}}e^{\sqrt q\,r\,(e_1\cdot\omega)}\,d\sigma(\omega),
\]
by uniqueness for the radial equation with data $w(0)=1$, $w'(0)=0$.
Indeed, integrating $r^{k-1}w'=q\int_0^r\tau^{k-1}w(\tau)\dd\tau$
gives a Volterra equation near zero and then ordinary ODE uniqueness
away from zero; for $k=1$ the solution is even. Since
$\phi_q(0)=1$ and
$\phi_q''(r)=q\,|S^{k-1}|^{-1}\int_{S^{k-1}}(e_1\cdot\omega)^2
e^{\sqrt q\,r\,(e_1\cdot\omega)}\,d\sigma(\omega)>0$, we get
$q^{-1}-v_q=(q^{-1}-v_q(0))\phi_q$ in $B_R$, and
$-v_q'=(q^{-1}-v_q(0))\phi_q'$ is increasing on $(0,R]$.
Outside,
\[
(r^{k-1}(-v_q'))'=-qr^{k-1}v_q<0.
\]

Integrating these two monotonicity statements against the positive weight $\kappa_{k,s}m_\lambda(q)\dd q$ gives the assertions on $-V'$; they extend to $r=R$ by continuity of $V'$. Concavity and convexity follow from these first-derivative comparisons. For $k=1$, differentiate the interval
convolution directly.

For $k=1$, the bound follows directly from $V(2R)\le\int_R^{3R}u^{-1-s}\dd u\le R^{-s}\ln3$. For $k\ge2$, viewed from $2Re_1$, the ball $B_R$ lies in the cone over the spherical cap $C\subset S^{k-1}$ of angular radius $\pi/6$ around $-e_1$, at distances between $R$ and $3R$. The normalized measure of a cap of angular radius $\varphi\le\pi/2$ is at most $\frac12\sin^{k-1}\varphi$ (substitute $\sin\vartheta=\sin\varphi\,\sin\vartheta'$ in $\int_0^\varphi\sin^{k-2}\vartheta\dd\vartheta$), so $|C|\le k\omega_k2^{-k}$.
 Since $K_\lambda(|z|)\le|z|^{-k-s}$,
\[
V(2R)\le|C|\int_R^{3R}r^{-1-s}\dd r=|C|\,R^{-s}\,\frac{1-3^{-s}}{s}\le k\omega_k2^{-k}\ln3\;R^{-s}.
\qedhere
\]
\end{proof}
\subsection{Curvature and endpoint identities}
\begin{lemma}\label{lem:h}
The following properties hold.
\begin{enumerate}
\item[(i)] The function $h$ is finite, positive and decreasing on $(0,\infty)$, tends to $+\infty$ as $r\downarrow0$, and satisfies
\begin{equation}\label{eq:hbound}
\begin{aligned}
h(r)&\le k\eta r^{-1},\\
Ah(R)&\le\int_0^\infty U_t(A)\dd\nu_{\rm rem}(t)
\le A R^{-1} k\eta .
\end{aligned}
\end{equation}
\item[(ii)] It is of class $C^1$ on $(0,\infty)$, and $-h'$ is positive and strictly decreasing.
\item[(iii)] If $k=1$, then
\[
-h'(r)=4\int_0^\infty G_t^{(1)}(2r)\dd\nu_{\rm rem}(t).
\]
\item[(iv)] If $k=2$, then
\begin{equation}\label{eq:curvature2}
-h'(r)\le 2\eta r^{-2},\quad
h(r)\le  2\eta r^{-1};
\end{equation}
if $k\ge3$, then
\begin{equation}\label{eq:curvaturehigh}
-h'(r)\le (k-1)\eta r^{-2},\quad
h(r)\le  (k-1)\eta r^{-1}.
\end{equation}
For $k\ge3$, $r\mapsto r^{k-1}(-h'(r))$ is strictly increasing.
\end{enumerate}
\end{lemma}
\begin{proof}
Scaling first gives the fixed-domain formula
\[
h(r)=2\kappa_{k,s} r^{-s}\int_0^\infty e^{-\lambda r^2\tau}\delta_\tau(1)\tau^{-1-s/2}\dd\tau.
\]

On any compact positive radius interval, $\delta_\tau(1)\le\min\{1/2,k\sqrt\tau/(2\sqrt\pi)\}$ dominates the integrand and its radius derivative at zero; exponential decay dominates them at infinity. 
Thus $h\in C^1$.
Differentiating this fixed-domain formula and returning to $t=r^2\tau$ gives
\begin{equation}\label{eq:hintegrated}
\begin{aligned}
-h'(r)&=\frac{4\kappa_{k,s}}{r}\int_0^\infty e^{-\lambda t}\delta_t(r)
 [\lambda t^{-s/2}+\frac{s}{2}t^{-1-s/2}]\dd t\\
&\le\frac{2k\kappa_{k,s}}{\sqrt\pi r^2}\lambda^{(s-1)/2}
 [\Gamma((3-s)/2)+(s/2)\Gamma((1-s)/2)]
\\&=k\eta r^{-2}.
\end{aligned}
\end{equation}
The Gamma recurrence reduces the bracket to $\Gamma((1-s)/2)/2$.

By \eqref{eq:deltamon}, \eqref{eq:heatloss} and \eqref{eq:eta} we obtain the first bound in \eqref{eq:hbound}; $h$ is decreasing by \eqref{eq:deltamon}, and $h(r)\to\infty$ as $r\to0$ by monotone convergence since $\int_0^\infty \frac12\dd\nu_{\rm rem}(t)=\infty$. By \eqref{eq:dJ} and Tonelli, 
\[
\int_0^\infty U_t(A)\dd\nu_{\rm rem}(t)=\int_0^Ah(r(v))\dd v\ge A\,h(R)
,
\]
 and the upper bound for $\int_0^\infty U_t(A)\dd\nu_{\rm rem}(t)$ follows from the first inequality in \eqref{eq:heatloss} in the same way.

Differentiating $\delta_t(r)=\frac12-\int_{B_r}G_t(re_1-w)\dd w$ and using the divergence theorem,
\begin{equation}\label{eq:bt}
b_t(r):=-\partial_r\delta_t(r)=\int_{\partial B_r}G_t(re_1-w)\,\big(1-e_1\cdot \nu_w\big)\dd\sigma(w)\ge0 .
\end{equation}
If $k=1$, $\partial B_r=\{\pm r\}$ and $b_t(r)=2G^{(1)}_t(2r)$. If $k\ge2$, write $w=r\nu$, $c=e_1\cdot\nu$, so that $|re_1-w|^2=2r^2(1-c)$; slicing the sphere by the level sets of $c$ and substituting the chord length $y=r\sqrt{2(1-c)}\in(0,2r)$ gives
\begin{equation}\label{eq:chord}
b_t(r)=\frac{|S^{k-2}|}{2r^2}\int_0^{2r}\frac{e^{-y^2/(4t)}}{(4\pi t)^{k/2}}\;y^{k}\Big(1-\frac{y^2}{4r^2}\Big)^{\frac{k-3}{2}}\dd y .
\end{equation}
The surface formula \eqref{eq:bt} gives a compact-radius majorant $C\sqrt t$ at small times (split the sphere at polar angle $\pi/2$; for $k=1$ use the exact Gaussian), with exponential domination at large times in the weighted remainder integral. This also justifies the derivative exchange and $-h'(r)=2\int_0^\infty b_t(r)\dd\nu_{\rm rem}(t)$.

For $k\ge3$ the last factor is at most $1$, and 
\[
|S^{k-2}|\int_0^\infty(4\pi t)^{-(k-1)/2}e^{-y^2/(4t)}y^k\dd y=2t(k-1)
\]
 is the second moment of the Gaussian of $\R^{k-1}$.
Thus, for $k\ge3$,
\begin{equation}\label{eq:bbound}
b_t(r)\le(k-1)tG_t^{(1)}(0)r^{-2}=\frac{k-1}{2\sqrt\pi}\frac{\sqrt t}{r^2}.
\end{equation}

For $k=2$, \eqref{eq:hintegrated} gives the time-integrated bound in \eqref{eq:curvature2}. For $k\ge3$, integrating \eqref{eq:bbound} preserves the better coefficient $k-1$; for $k=1$ the exact formula proves the dimension-one assertion. Finally $h(r)=\int_r^\infty(-h'(\sigma))\dd\sigma$ gives the respective bounds for $h$.

\emph{Strict monotonicity.} For $r>0$ let 
\[
\phi(r)=\kappa_{k,s}\int_0^\infty e^{-\lambda r^2\tau}(4\pi\tau)^{-k/2}e^{-1/(4\tau)}\tau^{-1-s/2}\dd\tau
;
\]
 then $\phi$ is strictly decreasing, and the substitution $t=r^2\tau$ shows $\int_0^\infty G_t(r e_1)\dd\nu_{\rm rem}(t)=r^{-k-s}\phi(r)$. By \eqref{eq:bt} with $w=r\nu$ and Tonelli,
\[
-h'(r)=2r^{-1-s}\int_{S^{k-1}}|\nu-e_1|^{-k-s}\,(1-e_1\cdot\nu)\,\phi\big(r|\nu-e_1|\big)\dd\sigma(\nu)\qquad(k\ge2),
\]
and $-h'(r)=4(2r)^{-1-s}\phi(2r)$ for $k=1$; both are strictly decreasing in $r$.
For $k\ge3$, after multiplication by $r^{k-1}$, the factor $r^{k-3}(1-y^2/(4r^2))^{(k-3)/2}$ is nondecreasing on $r>y/2$; the positive integration interval expands strictly. In particular, \eqref{eq:chord} shows that
\[
r^{k-1}b_t(r)=\frac{|S^{k-2}|}{2}r^{k-3}
\int_0^{2r}\frac{e^{-y^2/(4t)}}{(4\pi t)^{k/2}}y^k\left(1-\frac{y^2}{4r^2}\right)^{(k-3)/2}\dd y
\]
 is strictly increasing in $r$, which gives the asserted weighted strict increase.
\end{proof}
The next lemma links $V$ and $h$ at $r=R$ and evaluates
$\mathcal G-\theta\mathcal L$ at $a=0$; this is the source of condition
$(\mathrm I_\theta)$.
\begin{lemma}\label{lem:formulas}
The boundary value of the potential is
\begin{equation}\label{eq:VR}
V(R)=\frac12\left[\ell-\left(1-\frac sk\right)\frac{b_{k,s}}{\omega_k}R^{-s}+h(R)\right].
\end{equation}

For every $\theta\ge1$,
\begin{equation}\label{eq:endpoint}
\begin{aligned}
\mathcal G(0)-\theta\mathcal L(0)
&=\frac12\left[A\ell-\left(1+\frac sk\right)b_{k,s}R^{k-s}\right]\\
&\quad-(\theta-1)\int_0^\infty U_t(A)\dd\nu_{\rm rem}(t)
+\left(\theta-\frac12\right)Ah(R),
\end{aligned}
\end{equation}
and
\[
\mathcal G(0)\ge\frac12\left[A\ell-\left(1+\frac sk\right)b_{k,s}R^{k-s}\right].
\]

Together with \eqref{eq:cone} the boundary identity gives
\begin{equation}\label{eq:conedifference}
V(R)-V(2R)\ge\frac12\left[\ell-\left(1-\frac sk\right)\frac{b_{k,s}}{\omega_k}R^{-s}\right]-k2^{-k}\omega_k\ln3\,R^{-s}.
\end{equation}
\end{lemma}
\begin{proof}
By scaling, $\delta_t(r)=\delta_1(r/\sqrt t)$, so $2\kappa_{k,s}\int_0^\infty\delta_t(r)t^{-1-s/2}\dd t=C\,r^{-s}$ for a constant $C$. Integrating $U_t'(v)=2\delta_t(r(v))$ over $v\in(0,p)$ and using \eqref{eq:P} gives 
\[
b_{k,s}r(p)^{k-s}=\int_0^pC\,r(v)^{-s}\dd v
,
\qquad\text{hence}\quad
C\,r(p)^{-s}=\frac{d}{dp}[b_{k,s}r(p)^{k-s}]=\left(1-\frac sk\right)\frac{b_{k,s}}{\omega_k}r(p)^{-s}
.
\]
 Since $u_t(R)=\frac12-\delta_t(R)$, it follows that
\[
V(R)=\frac12\ell-\int_0^\infty\delta_t(R)\dd\nu_{\rm gap}(t)
=\frac12\ell-\frac12\left[\left(1-\frac sk\right)\frac{b_{k,s}}{\omega_k}R^{-s}-h(R)\right].
\]

By \eqref{eq:dJ}, $\int_0^Au_t(r(v))\dd v=J_t(A,A)=A-U_t(A)$, so \eqref{eq:P}, \eqref{eq:WFint} and \eqref{eq:VR} give
\[
\begin{aligned}
\int_0^AV(r(v))\dd v
&=A\ell-b_{k,s}R^{k-s}+\int_0^\infty U_t(A)\dd\nu_{\rm rem}(t),\\
\mathcal G(0)&=A\ell-b_{k,s}R^{k-s}+\int_0^\infty U_t(A)\dd\nu_{\rm rem}(t)\\
&\quad-\frac A2\left[\ell-\left(1-\frac sk\right)\frac{b_{k,s}}{\omega_k}R^{-s}+h(R)\right],\\
\mathcal L(0)&=\int_0^\infty U_t(A)\dd\nu_{\rm rem}(t)-A\,h(R),
\end{aligned}
\]
and \eqref{eq:endpoint} follows because $1-\left(1-\frac sk\right)/2=\frac12(1+\frac sk)$. The lower bound for $\mathcal G(0)$ follows from its displayed formula and $\int_0^\infty U_t(A)\dd\nu_{\rm rem}(t)\ge A\,h(R)$.
Inequality \eqref{eq:conedifference} follows from \eqref{eq:VR}, \eqref{eq:cone} and $h\ge0$.
\end{proof}
\subsection{Interior and exterior criteria}

For $\theta\ge1$ put $Q_\theta=\mathcal G-\theta\mathcal L$ on $[0,\infty)$ and,
in the radius variable,
\[
q_\theta(r)=\theta[h(r)-h(R)]-[V(r)-V(R)]\qquad(r>0).
\]
By \eqref{eq:WFint}, $Q_\theta(a)=\int_A^aq_\theta(r(v))\,dv$; thus $Q_\theta$ is
$C^1$ on $(0,\infty)$ with $Q_\theta'(a)=q_\theta(r(a))$, and $Q_\theta(A)=0$.
\begin{lemma}\label{lem:inside}
Let $\theta\ge1$. If $Q_\theta(0)\ge0$, then $Q_\theta>0$ on $(0,A)$.
\end{lemma}
\begin{proof}
On $(0,R]$ we have $q_\theta'=-V'+\theta h'$, which is strictly increasing by Lemma~\ref{lem:V}(ii) and Lemma~\ref{lem:h}(ii). Hence $q_\theta$ is strictly convex on $(0,R]$, $q_\theta(R)=0$, and $q_\theta(r)\to+\infty$ as $r\to0$ by Lemma~\ref{lem:h}(i). So either $q_\theta>0$ on $(0,R)$, or there is $r_c\in(0,R)$ with $q_\theta>0$ on $(0,r_c)$ and $q_\theta<0$ on $(r_c,R)$.

In the first case $Q_\theta$ is strictly increasing on $[0,A]$, so $Q_\theta(0)<Q_\theta(A)=0$, contrary to the assumption.

In the second case $Q_\theta$ is strictly increasing on $[0,A_c]$ and strictly decreasing on $[A_c,A]$, where $A_c=\omega_k r_c^k$; hence $Q_\theta>Q_\theta(0)\ge0$ on $(0,A_c]$ and $Q_\theta>Q_\theta(A)=0$ on $[A_c,A)$.
\end{proof}
\begin{lemma}\label{lem:outside}
Let $\theta\ge1$. Each of the following conditions implies $q_\theta>0$ on $(R,\infty)$, hence $Q_\theta>0$ on $(A,\infty)$:
\begin{enumerate}
\item[(a)] $k=1$ and $R\sqrt\lambda\ge\ln(1+2\sqrt\theta)$;
\item[(b)] $k\ge2$ and $V(R)-V(2R)\ge2\theta\eta/R$ for $k=2$,
or $V(R)-V(2R)\ge(k-1)\theta\eta/R$ for $k\ge3$;
\item[(c)] $k\ge3$ and $V(R)\ge \theta\,h(R)$.
\end{enumerate}
\end{lemma}
\begin{proof}
For $r>R$, $q_\theta(r)=[V(R)-V(r)]-\theta[h(R)-h(r)]=\int_R^r(-V'+\theta h')$.

(a) By Lemma~\ref{lem:V}(iii) and Lemma~\ref{lem:h}(iii), for $y>0$,
\begin{align*}
&(-V'+\theta h')(R+y)=\kappa_{k,s}\int_0^\infty I_t(y)\,\frac{\dd t}{t^{1+s/2}},\\
&I_t(y)=(1-e^{-\lambda t})[G^{(1)}_t(y)-G^{(1)}_t(2R+y)]-4\theta e^{-\lambda t}G^{(1)}_t(2R+2y).
\end{align*}
Since $G^{(1)}_t(2R+y)$ and $G^{(1)}_t(2R+2y)$ are at most $e^{-R^2/t}G^{(1)}_t(y)$,
\[
I_t(y)\ge G^{(1)}_t(y)\,e^{-\lambda t-R^2/t}\big[(e^{\lambda t}-1)(e^{R^2/t}-1)-4\theta\big].
\]
The function $\psi(\xi)=\ln(e^{e^\xi}-1)$ is convex, because $\psi'(\xi)=e^\xi/(1-e^{-e^\xi})$ is increasing; hence 
\[
(e^{\lambda t}-1)(e^{R^2/t}-1)\ge(e^{R\sqrt\lambda}-1)^2
,
\]
and the bracket is at least $(e^{R\sqrt\lambda}-1)^2-4\theta\ge0$, with strict inequality for $\lambda t\ne R^2/t$. Thus $-V'+\theta h'>0$ on $(R,\infty)$.

(b) Put $\epsilon=2\eta$ when $k=2$, and $\epsilon=(k-1)\eta$ when $k\ge3$. By Lemma~\ref{lem:h}(iv), 
\[
h(R)-h(r)=\int_R^r(-h'(\sigma))\dd\sigma\le\epsilon(r-R)/(Rr)
.
\]
 If $R<r\le2R$, convexity of $V$ on $[R,2R]$ (Lemma~\ref{lem:V}(ii)) gives 
\[
V(R)-V(r)\ge\frac{r-R}{R}[V(R)-V(2R)]
,
\]
 so that
\[
q_\theta(r)\ge\frac{r-R}{R}\big[V(R)-V(2R)-\theta\epsilon/r\big]>0
.
\]
 If $r\ge2R$, then $V(R)-V(r)\ge V(R)-V(2R)\ge\theta\epsilon/R>\theta[h(R)-h(r)]$.

(c) By Lemma~\ref{lem:V}(ii) and Lemma~\ref{lem:h}(iv), 
$r^{k-1}q_\theta'(r)=r^{k-1}(-V'+\theta h')(r)$
 is strictly decreasing on $(R,\infty)$. Hence there is $r_*\in[R,\infty]$ such that $q_\theta$ is strictly increasing on $[R,r_*]$ and strictly decreasing on $[r_*,\infty)$. Recall that $q_\theta(R)=0$ and $q_\theta(r)\to V(R)-\theta h(R)\ge0$ as $r\to\infty$ (Lemma~\ref{lem:V}(i) and Lemma~\ref{lem:h}(i)).

If $r_*=R$, then $q_\theta$ would be strictly decreasing on $[R,\infty)$ starting from $0$, which contradicts the sign of its limit. Hence $r_*>R$, $q_\theta>0$ on $(R,r_*]$, and on $[r_*,\infty)$ the function $q_\theta$ decreases strictly to a nonnegative limit, so it stays positive there as well.
\end{proof}
\begin{proof}[Proof of Proposition~\ref{prop:scalar}]
Put $x=R\sqrt\lambda$, so that $R=x/\sqrt\lambda$ and $\frac{b_{k,s}}{\omega_k}R^{-s}=\frac{b_{k,s}}{\omega_k}\lambda^{s/2}x^{-s}$.

\emph{The range $a<A$.} Equation~\eqref{eq:endpoint}, the upper bound in \eqref{eq:hbound}, and $(\mathrm I_\theta)$ imply
\[
\frac{\mathcal G(0)-\theta\mathcal L(0)}A
\ge\frac12\left[\ell-\left(1+\frac sk\right)\frac{b_{k,s}}{\omega_k}R^{-s}\right]-(\theta-1)\frac1A\int_0^\infty U_t(A)\dd\nu_{\rm rem}(t)
\ge0.
\]

Lemma~\ref{lem:inside} gives $Q_\theta>0$ on $(0,A)$. If $\theta=1$, \eqref{eq:endpoint} and $(\mathrm I_1)$ give
\[
(\mathcal G-\mathcal L)(0)\ge\frac12Ah(R)>0.
\]

\emph{The range $a>A$.} For $k=1$, $(\mathrm{II}_\theta)$ is condition (a) of Lemma~\ref{lem:outside}.
For $k=2$, \eqref{eq:conedifference} gives
\[
V(R)-V(2R)\ge\frac12\left[\ell-(1-s/2)\frac{b_{2,s}}{\omega_2}R^{-s}\right]
-\frac12\omega_2\ln3\,R^{-s}.
\]

After substituting $R=x/\sqrt\lambda$ and \eqref{eq:gapmass}, $(\mathrm{II}_\theta)$ says exactly that this is at least $2\theta\eta/R$. Thus condition (b) holds.
For $k\ge3$, \eqref{eq:VR} yields
\[
V(R)-\theta h(R)=\frac12\left[\ell-\left(1-\frac sk\right)\frac{b_{k,s}}{\omega_k}R^{-s}\right]
-\frac12(2\theta-1)h(R).
\]

Using the bound $h(R)\le(k-1)\eta/R$ from \eqref{eq:curvaturehigh},
together with \eqref{eq:gapmass} and $R=x/\sqrt\lambda$, condition
$(\mathrm{II}_\theta)$ makes the right-hand side nonnegative and yields condition (c).
In every branch Lemma~\ref{lem:outside} gives $Q_\theta>0$ on $(A,\infty)$. At $a=A$ both gaps are zero. This proves the proposition on $[0,\infty)$.
\end{proof}

\section{Proofs of the main results}\label{sec:proofs}
\subsection{Uniqueness}\label{sec:proofA}
The inequality follows from Propositions~\ref{prop:reduction} and
\ref{prop:scalar}. The equality case uses the following rigidity lemma.
\begin{lemma}\label{lem:rigidity}
Let $E\subset M\times\mathbb R^k$ be measurable with $|E|=|M|A$. Suppose that $a_E=A$ almost everywhere and
$\cP_s(E)-\cP_s(M\times B_R)=0$.
Then $E=M\times B_R(z_0)$ modulo null sets for some $z_0\in\R^k$.
\end{lemma}
\begin{proof}
Proposition~\ref{prop:reduction} gives $D(\bar\rho_E)=0$, since $\mathcal L(A)=0$.
Every $D_t(\bar\rho_E)$ is nonnegative and the weight
$(1-e^{-\lambda t})t^{-1-s/2}$ is positive, so this interaction deficit
vanishes for almost every $t>0$. Fix one such time, put $\rho=\bar\rho_E$, and let $\rho^*$ be its
symmetric decreasing rearrangement. The Riesz inequality
\eqref{eq:rieszfn} and Lemma~\ref{lem:radialcomparison} give
$\iint\rho G_t\rho\le\iint\rho^*G_t\rho^*\le J_t(A,A)$.

The function $G_t*\chi_{B_R}$ is strictly radially decreasing. Indeed, take $z_i=r_i e_1$ with $0\le r_1<r_2$, let $H_-$ be the half-space containing the origin and bounded by the perpendicular bisector of $z_1,z_2$, and let $\sigma$ be reflection in its boundary. Pairing $w$ and $\sigma w$ gives
\[
(G_t*\chi_{B_R})(z_1)-(G_t*\chi_{B_R})(z_2)
=\int_{H_-}[G_t(z_1-w)-G_t(z_2-w)]\cdot[\chi_{B_R}(w)-\chi_{B_R}(\sigma w)]\dd w>0.
\]
The first bracket is strictly positive in the half-space, and the second is nonnegative and positive on a set of positive measure.

Lemma~\ref{lem:radialcomparison} now applies to $K=G_t$ and $\rho^*$, a radial nonincreasing density in $[0,1]$ of mass $A$; equality forces $\rho^*=\chi_{B_R}$.
Equimeasurability then gives $\int\rho(1-\rho)=0$, hence $\rho=\chi_F$ for a measurable set $F$ of volume $A$.

It remains to use equality in the first inequality. Apply Burchard--Hajaiej \cite[Theorem~2]{BH} to $\chi_F,\chi_F$, the integrand $(u,v)\mapsto uv$, and the Gaussian kernel. The kernel is positive and strictly decreasing, $(u,v)\mapsto uv$ is strictly supermodular, and both indicators are nonconstant with finite positive level sets. Equality forces $F$ to agree with a translate $B_R(z_0)$ up to null sets. Thus $\bar\rho_E=\chi_{B_R(z_0)}$, and \eqref{eq:alpharho} gives $|E\triangle(M\times B_R(z_0))|=0$.
\end{proof}
\begin{proof}[Proof of Theorem~\ref{thm:A}]
If $\mathcal P_s(E)=+\infty$ there is nothing to prove. Otherwise
Proposition~\ref{prop:reduction} and Proposition~\ref{prop:scalar} with
$\theta=1$ give
\begin{equation}\label{eq:absorption}
\begin{aligned}
\mathcal P_s(E)-\mathcal P_s(M\times B_R)
&\ge c_{m,k,s}|M|\left(D(\bar\rho_E)-\fint_M\mathcal L(a_E)\,d\mu\right)\\
&\ge c_{m,k,s}\int_M\bigl[\mathcal G(a_E)-\mathcal L(a_E)\bigr]\,d\mu\ge0.
\end{aligned}
\end{equation}
Equation~\eqref{eq:absorption} proves the asserted inequality. If equality holds, then
$\int_M[\mathcal G(a_E)-\mathcal L(a_E)]\,d\mu=0$. By
Proposition~\ref{prop:scalar}, including the case $a=0$, the function
$\mathcal G-\mathcal L$ is positive on $[0,\infty)\setminus\{A\}$; hence
$a_E=A$ almost everywhere, and Lemma~\ref{lem:rigidity} shows that $E$ is a
translated cylinder up to a null set. Conversely, $\mathcal P_s$ is invariant
under translations in $\mathbb R^k$, so every translated cylinder attains equality.
\end{proof}

\subsection{Stability}\label{sec:proofB}
\subsubsection*{Absorption and volume coercivity}\label{sec:scalarretention}
Assume $R\sqrt\lambda\ge T_{k,s}(2)$. If $\mathcal P_s(E)=+\infty$, the bounds are immediate. Otherwise, Proposition~\ref{prop:scalar} with $\theta=2$ and Proposition~\ref{prop:reduction} give
\[
\fint_M\mathcal L(a_E)\,d\mu
\le\frac12\fint_M\mathcal G(a_E)\,d\mu
\le\frac12 D(\bar\rho_E).
\]
Consequently,
\begin{equation}\label{eq:scalarretention}
\mathcal P_s(E)-\mathcal P_s(M\times B_R)
\ge\frac{c_{m,k,s}|M|}{2}D(\bar\rho_E)
\ge\frac{c_{m,k,s}}2\int_M\mathcal G(a_E)\,d\mu.
\end{equation}

\begin{lemma}\label{lem:Wlower}
Let
\[
c=\frac12\left[\ell-\left(1+\frac sk\right)\frac{b_{k,s}}{\omega_k}R^{-s}\right]
\]
and assume $c\ge0$. Then
\[
\mathcal G(a)\ge\frac{cA}{2k}\frac{(a/A-1)^2}{a/A+1}\qquad(a\ge0).
\]
\end{lemma}
\begin{proof}
Write $u=a/A$.
\emph{Case $a\le A$.} By Lemma~\ref{lem:V}(ii), $V$ is concave on $[0,R]$, so $V(r)-V(R)\ge(1-r/R)(V(0)-V(R))$, and 
$V(0)-V(R)\ge \mathcal G(0)/A\ge c$ by \eqref{eq:WFint} and Lemma~\ref{lem:formulas}. Since $1-(v/A)^{1/k}\ge(1-v/A)/k$,
\[
\mathcal G(a)=\int_a^A[V(r(v))-V(R)]\dd v
\ge\frac ck\int_a^A(1-v/A)\dd v
=\frac{cA}{2k}(1-u)^2\ge\frac{cA}{2k}\frac{(u-1)^2}{u+1}.
\]
\emph{Case $a\ge A$.} We first establish the coefficient bound for the exterior estimate.
Here $\psi$ denotes the digamma function.
From the formula for $b_{k,s}$,
\[
\frac{s}{k}\,\frac{b_{k,s}}{\omega_k}=2^{-s}\pi^{(k-1)/2}\Gamma\left(\frac{1-s}2\right)/\Gamma\left(\frac{k-s}2+1\right),
\]
which equals $\omega_k$ at $s=0$. Its logarithmic derivative in $s$ is
\[
-\ln2-\frac12\psi\left(\frac{1-s}2\right)+\frac12\psi\left(\frac{k-s}2+1\right)
\ge-\ln2-\frac12\psi(\frac12)+\frac12\psi(1)=0,
\]
since $\psi$ is increasing and $\psi(1)-\psi(\frac12)=2\ln2$. Finally $k2^{-k}\ln3\le\frac12\ln3<1$. Thus
\[
\frac sk\frac{b_{k,s}}{\omega_k}\ge\omega_k>k2^{-k}\omega_k\ln3.
\]

By Lemma~\ref{lem:formulas},
\[
V(R)-V(2R)\ge c+\big[\frac{s}{k}\,\frac{b_{k,s}}{\omega_k}-k2^{-k}\omega_k\ln3\big]R^{-s}\ge c.
\]

Convexity and monotonicity of $V$ on $[R,\infty)$ (Lemma~\ref{lem:V}(ii)) give 
\[
V(R)-V(r)\ge c\min\{1,r/R-1\}
,
\]
 and for $v\ge A$ the mean value theorem gives 
$(v/A)^{1/k}-1\ge(v/A-1)/(k\,v/A)
,$
 whose right-hand side is at most $1$. Hence
\[
\mathcal G(a)=\int_A^a[V(R)-V(r(v))]\dd v
\ge cA\int_1^u\frac{\tau-1}{k\tau}\dd\tau
\ge\frac{c A}{2k}\,\frac{(u-1)^2}{u}\ge\frac{c A}{2k}\frac{(u-1)^2}{u+1}.
\qedhere
\]
\end{proof}

The first estimate of Theorem~\ref{thm:B} follows from
\eqref{eq:scalarretention} and Lemma~\ref{lem:Wlower}. The second requires a
quantitative form of the Riesz inequality, to which we now turn.

\subsubsection*{The quantitative Riesz inequality and the asymmetry}\label{sec:highconversion}

For ball-kernel interactions let $B\subset\mathbb R^k$ be a centered ball and $\rho:\mathbb R^k\to[0,1]$ be integrable. Set
\[
I_B(\rho)=\iint\rho(z)\chi_B(z-w)\rho(w)\dd z\dd w.
\]
Frank and Lieb stated the following density form of the quantitative Riesz inequality in \cite[Theorem~3]{FLflock} and proved it in the supplementary note \cite{FLnote}, building on Christ's quantitative Riesz--Sobolev theorem \cite{Christ}.
\begin{citedtheorem}\label{thm:FL}
For every $k\ge1$ and $0<\delta\le\frac12$ there is $c_{k,\delta}>0$ such that for all balls $B_r\subset\R^k$ centered at the origin and all $\rho\in L^1(\R^k)$ with $0\le\rho\le1$, $\int\rho=|B_R|$ and $\delta\le r/(2R)\le1-\delta$,
\[
I_{B_r}(\rho)\le I_{B_r}(\chi_{B_R})-\frac{c_{k,\delta}}2\,\Big(\inf_{z_0\in\R^k}\|\rho-\chi_{B_R(z_0)}\|_1\Big)^2 .
\]
\end{citedtheorem}

Here $c_{k,\delta}$ is the constant in the cited density theorem, whose interaction is $\frac12 I_B$ and whose asymmetry is the $L^1$ distance from a translated ball divided by $2A$, hence the coefficient is $c_{k,\delta}/2$.
Frank and Lieb use a layer-cake decomposition to pass from quantitative Riesz inequalities to fractional-perimeter stability \cite{FLflock}. Here we apply that passage to the integrable kernel produced by spectral splitting, on a fixed radius-ratio interval.
\begin{lemma}\label{lem:layer}
Let $\rho:\mathbb R^k\to[0,1]$ with $\int\rho=A=|B_R|$. Then
\[
D(\rho)\ge\frac{c_{k,1/4}(1-e^{-\lambda R^2/4})\kappa_{k,s}}{4(4\pi)^{k/2}}R^{-k-s}\left(\inf_{z_0\in\mathbb R^k}\|\rho-\chi_{B_R(z_0)}\|_1\right)^2.
\]
\end{lemma}
\begin{proof}
Let $\rho^*$ be the symmetric decreasing rearrangement of $\rho$.
By \eqref{eq:rieszfn} and the first assertion of
Lemma~\ref{lem:radialcomparison}, applied with $K=\chi_{B_r}$
and $f=\rho^*$,
$I_{B_r}(\rho)\le I_{B_r}(\rho^*)\le I_{B_r}(\chi_{B_R})$.
The function $K_\lambda$ is positive, decreasing and $C^1$ on $(0,\infty)$ with $K_\lambda(\infty)=0$, so 
\[
\begin{aligned}
K_\lambda(|z|)&=\int_0^\infty\chi_{B_r}(z)(-K_\lambda^\prime(r))\dd r,\\
D(\rho)&=\int_0^\infty[I_{B_r}(\chi_{B_R})-I_{B_r}(\rho)](-K_\lambda^\prime(r))\dd r.
\end{aligned}
\]
All brackets are nonnegative; on $r\in[R/2,3R/2]$ we apply Theorem~\ref{thm:FL} with $\delta=\frac14$. This gives
\begin{equation}\label{eq:layerretention}
D(\rho)\ge\frac{c_{k,1/4}}2\left[K_\lambda(R/2)-K_\lambda(3R/2)\right]\left(\inf_{z_0\in\mathbb R^k}\|\rho-\chi_{B_R(z_0)}\|_1\right)^2.
\end{equation}
 To estimate the kernel difference, restrict the time integral defining $K_\lambda$ to $t\in[R^2/4,R^2]$. There 
\[
1-e^{-\lambda t}\ge1-e^{-\lambda R^2/4},\qquad (4\pi t)^{-k/2}\ge(4\pi R^2)^{-k/2},
\]
and, since $u\mapsto e^{-u}-e^{-9u}$ is increasing on
$[\frac1{16},\frac14]$,
\[
e^{-R^2/(16t)}-e^{-9R^2/(16t)}
\ge e^{-1/16}-e^{-9/16}
=e^{-1/16}\bigl(1-e^{-1/2}\bigr).
\]
Moreover,
\[
\int_{R^2/4}^{R^2}t^{-1-s/2}\dd t
=\frac2s(2^s-1)R^{-s}\ge2\ln2\,R^{-s}.
\]
Finally, \(2e^{-1/16}(1-e^{-1/2})\ln2>1/2\).
The preceding estimates prove
\[
K_\lambda(R/2)-K_\lambda(3R/2)\ge\frac{(1-e^{-\lambda R^2/4})\kappa_{k,s}}{2(4\pi)^{k/2}}R^{-k-s},
\]
which combines with \eqref{eq:layerretention} to give the stated conclusion.
\end{proof}

\subsubsection*{Assembly and normalization}
The coefficients in Theorem~\ref{thm:B} are given by
\begin{equation}\label{eq:Cchoice}
\begin{aligned}
C_1(k,s,x)&=\frac1{8k}\left[
 \frac{2\omega_k\kappa_{k,s}\Gamma(1-s/2)}{s b_{k,s}}x^s
 -\left(1+\frac sk\right)\right],\qquad x\ge T_{k,s}(2),\\
C_2(k,s)&=\frac{c_{k,1/4}(1-e^{-T_{k,s}(2)^2/4})
 \kappa_{k,s}\omega_k^2}{8(4\pi)^{k/2}b_{k,s}}.
\end{aligned}
\end{equation}
Here $c_{k,1/4}$ is the constant in Theorem~\ref{thm:FL}.

\begin{proof}[Proof of Theorem~\ref{thm:B}]
If $\mathcal P_s(E)=+\infty$, both conclusions are immediate; assume $\mathcal P_s(E)<\infty$, so that \eqref{eq:scalarretention} holds.
Combining \eqref{eq:scalarretention} with Lemma~\ref{lem:Wlower}, whose hypothesis $c\ge0$ holds by $(\mathrm I_2)$ and \eqref{eq:gapmass}, gives
\[
\begin{aligned}
\mathcal P_s(E)-\mathcal P_s(M\times B_R)
&\ge\frac{c_{m,k,s}}2\int_M\mathcal G(a_E)\dd\mu\\
&\ge\frac{c_{m,k,s}A}{8k}
 \left[\ell-\left(1+\frac sk\right)\frac{b_{k,s}}{\omega_k}R^{-s}\right]
 \int_M\frac{(a_E/A-1)^2}{a_E/A+1}\dd\mu.
\end{aligned}
\]

Dividing the displayed volume bound by \eqref{eq:cylinder}, using $A=\omega_kR^k$ and \eqref{eq:gapmass}, gives the first coefficient in \eqref{eq:Cchoice}, since
\[
\begin{aligned}
\frac1{|M|}\int_M\frac{(a_E/A-1)^2}{a_E/A+1}\dd\mu
&=\frac1{|E|}\int_M\frac{(a_E-A)^2}{a_E+A}\dd\mu,\\
\frac{\omega_kR^s}{8kb_{k,s}}
 \left[\ell-\left(1+\frac sk\right)\frac{b_{k,s}}{\omega_k}R^{-s}\right]
&=C_1(k,s,R\sqrt\lambda).
\end{aligned}
\]

By \eqref{eq:alpharho}, $\inf_{z_0}\|\bar\rho_E-\chi_{B_R(z_0)}\|_1=A\alpha(E)$. Equation~\eqref{eq:scalarretention} and Lemma~\ref{lem:layer} yield
\[
\mathcal P_s(E)-\mathcal P_s(M\times B_R)
\ge\frac{c_{m,k,s}|M|}{2}\cdot\frac{c_{k,1/4}}2
 \cdot\frac{(1-e^{-\lambda R^2/4})\kappa_{k,s}}{2(4\pi)^{k/2}}
\cdot R^{-k-s}A^2\alpha(E)^2.
\]

Now $A^2R^{-k-s}=\omega_k^2R^{k-s}$. Divide by \eqref{eq:cylinder} and use $R\sqrt\lambda\ge T_{k,s}(2)$ and the increasing function $u\mapsto1-e^{-u^2/4}$. The resulting coefficient is at least $C_2(k,s)$ as defined in \eqref{eq:Cchoice}. This proves the second conclusion.
\end{proof}

\subsection{Flat tori}\label{sec:periodic}
The quotient volume equals the volume in one fundamental cell.
The following identity transfers the perimeter comparisons to periodic sets.
\begin{lemma}\label{lem:flat}
Let $\Lambda\subset\R^m$ be a lattice of full rank, $M=\R^m/\Lambda$ with the flat metric, and $Q$ a fundamental cell of $\Lambda$. For a measurable set $E\subset M\times\R^k$ of finite measure, with $\Lambda$-periodic lift $\widetilde E\subset\R^n$, $\mathcal P_s(E)=\mathcal P_s^\Lambda(\widetilde E)$, with the one-cell perimeter defined in \eqref{eq:periodic}.

\end{lemma}
\begin{proof}
For $x,y\in M$ with lifts $\tilde x,\tilde y\in\R^m$, the heat kernel of $M$ is given by the method of images,
\[
 h_t(x,y)=\sum_{\xi\in\Lambda}(4\pi t)^{-m/2}e^{-|\tilde x-\tilde y+\xi|^2/(4t)};
\]
the sum is $\Lambda$-periodic in each variable and solves the heat equation on $\R^m$ with initial datum the sum of Dirac masses at the points $\tilde y-\xi$. Writing $\widetilde X=(\tilde x,z)$ and $\widetilde Y=(\tilde y,w)$, we get
\[
 h^N_t(X,Y)=\sum_{\xi\in\Lambda}G_t^{(n)}(\widetilde X-\widetilde Y+(\xi,0)).
\]
Equation~\eqref{eq:subordination} in dimension $n$ gives
\[
 \kappa_{n,s}\int_0^\infty h_t^N(X,Y)t^{-1-s/2}\dd t
 =\sum_{\xi\in\Lambda}|\widetilde X-\widetilde Y+(\xi,0)|^{-n-s}.
\] Unfolding the lattice sum in the $Y$-integral over $N\setminus E$ yields $\mathcal P_s(E)=\mathcal P_s^\Lambda(\widetilde E)$ with the integral \eqref{eq:periodic}.\qedhere

\end{proof}
\begin{proof}[Proof of Corollary~\ref{cor:torus}]
Let $M=\mathbb R^m/\Lambda$. Lemma~\ref{lem:flat} identifies the two perimeters. Let the dual lattice be
\[
\Lambda^*=\{\xi\in\mathbb R^m:\xi\cdot v\in\mathbb Z\text{ for every }v\in\Lambda\}.
\]

The eigenfunctions are $e^{2\pi i\xi\cdot x}$, $\xi\in\Lambda^*$, with eigenvalues $4\pi^2|\xi|^2$, so
\[
\lambda_1(M)=4\pi^2\min_{\xi\in\Lambda^*\setminus\{0\}}|\xi|^2.
\]

Let $\xi$ attain the minimum and choose $x_0$ with $\xi\cdot x_0=\frac12$. Every $v\in\Lambda$ satisfies
\[
|v-x_0|\ge\frac{|\xi\cdot(v-x_0)|}{|\xi|}\ge\frac1{2|\xi|},
\]
so that $r_{\rm cov}(\Lambda)\ge\frac1{2|\xi|}$.

Thus
\[
\sqrt{\lambda_1(M)}\ge\frac{\pi}{r_{\rm cov}(\Lambda)}.
\]

The corollary's radius condition gives the hypothesis of Theorem~\ref{thm:A}; its comparison and equality transfer to the lift through the quotient identity. For $\Lambda=L\mathbb Z$, $r_{\rm cov}(\Lambda)=L/2$ and the spectral estimate is an equality.
The same identification transfers Theorem~\ref{thm:B} to $\Lambda$-periodic sets,
with $\lambda_1(M)=4\pi^2\min_{\xi\in\Lambda^*\setminus\{0\}}|\xi|^2$.
\end{proof}

\appendix
\section{The fractional perimeter of the unit ball}\label{app:ball}
We prove formula \eqref{eq:constants1} for $b_{k,s}=P_s^{(k)}(B_1^k)$ by a direct computation. Set $Q_k(r)=|B_1^k\setminus(B_1^k+re_1)|$. Using $\zeta=w-z$ and polar coordinates gives
\[
 P_s^{(k)}(B_1^k)=\int_{\R^k}\frac{|B_1^k\setminus(B_1^k-\zeta)|}{|\zeta|^{k+s}}\dd\zeta
 =k\omega_k\int_0^\infty\frac{Q_k(r)}{r^{1+s}}\dd r.
\]
Slice $B_1^k$ by lines parallel to $e_1$. The line through $z'\in B_1^{k-1}$ meets $B_1^k$ in a segment of length $2\sqrt{1-|z'|^2}$, of which a piece of length $\min\{r,2\sqrt{1-|z'|^2}\}$ lies outside $B_1^k+re_1$. Hence
\[
 Q_k(r)=\int_{B_1^{k-1}}\min\big\{r,2\sqrt{1-|z'|^2}\big\}\dd z',
\]
where for $k=1$ the right-hand side is read as $\min\{r,2\}$ and $\omega_0=1$. The function $Q_k$ is Lipschitz, satisfies $Q_k(0)=0$, is constant for $r\ge2$, and has derivative $Q_k'(r)=\omega_{k-1}(1-r^2/4)^{(k-1)/2}$ for $0<r<2$. Integration by parts followed by $r=2\sqrt\sigma$ gives
\[
 \int_0^\infty\frac{Q_k(r)}{r^{1+s}}\dd r=\frac{\omega_{k-1}}s\int_0^2r^{-s}\Big(1-\frac{r^2}4\Big)^{(k-1)/2}\dd r
 =\frac{\omega_{k-1}}{2^s\,s}\,\frac{\Gamma\big(\frac{1-s}2\big)\Gamma\big(\frac{k+1}2\big)}{\Gamma\big(\frac{k-s}2+1\big)} .
\]
Since $\omega_{k-1}\Gamma(\frac{k+1}2)=\pi^{(k-1)/2},\, \Gamma(\frac{k-s}2+1)=\frac{k-s}2\Gamma(\frac{k-s}2),$ this gives the formula for $b_{k,s}$.


\begin{thebibliography}{99}
\bibitem{AB}
F. Alcover and R. Bruera,
\emph{Periodic Delaunay cylinders with constant anisotropic nonlocal mean curvature},
preprint, arXiv:2602.18219 (2026).

\bibitem{Bruera}
R. Bruera,
\emph{On the shape of minimizers for the periodic nonlocal perimeter in $\R^2$},
preprint, arXiv:2602.18215 (2026).

\bibitem{BH}
A. Burchard and H. Hajaiej,
\emph{Rearrangement inequalities for functionals with monotone integrands},
J. Funct. Anal. \textbf{233} (2006), 561--582.

\bibitem{CCM}
X. Cabr\'e, G. Csat\'o, and A. Mas,
\emph{Existence and symmetry of periodic nonlocal-CMC surfaces via variational methods},
J. Reine Angew. Math. \textbf{804} (2023), 11--40.

\bibitem{CRS}
L. Caffarelli, J.-M. Roquejoffre, and O. Savin,
\emph{Nonlocal minimal surfaces},
Comm. Pure Appl. Math. \textbf{63} (2010), 1111--1144.

\bibitem{CFS2}
M. Caselli, E. Florit-Simon, and J. Serra,
\emph{Fractional Sobolev spaces on Riemannian manifolds},
Math. Ann. \textbf{390} (2024), 6249--6314.

\bibitem{Chavel}
I. Chavel,
\emph{Eigenvalues in Riemannian geometry},
Pure and Applied Mathematics, vol. 115, Academic Press, Orlando, FL, 1984.

\bibitem{Christ}
M. Christ,
\emph{A sharpened Riesz--Sobolev inequality},
preprint, arXiv:1706.02007 (2017).

\bibitem{CM}
G. Csat\'o and A. Mas,
\emph{Strict rearrangement inequalities: nonexpansivity and periodic Gagliardo seminorms},
Trans. Amer. Math. Soc. \textbf{378} (2025), 7163--7197.

\bibitem{DDDV}
J. D\'avila, M. del Pino, S. Dipierro, and E. Valdinoci,
\emph{Nonlocal Delaunay surfaces},
Nonlinear Anal. \textbf{137} (2016), 357--380.

\bibitem{FFMMM}
A. Figalli, N. Fusco, F. Maggi, V. Millot, and M. Morini,
\emph{Isoperimetry and stability properties of balls with respect to nonlocal energies},
Comm. Math. Phys. \textbf{336} (2015), 441--507.

\bibitem{FLflock}
R.~L.~Frank and E.~H.~Lieb,
\emph{Proof of spherical flocking based on quantitative rearrangement inequalities},
Ann. Sc. Norm. Super. Pisa Cl. Sci. (5) \textbf{22} (2021), 1241--1263.

\bibitem{FLnote}
R. L. Frank and E. H. Lieb,
\emph{A note on a theorem of M. Christ},
preprint, arXiv:1909.04598 (2019).

\bibitem{FS}
R. L. Frank and R. Seiringer,
\emph{Non-linear ground state representations and sharp Hardy inequalities},
J. Funct. Anal. \textbf{255} (2008), 3407--3430.

\bibitem{FMM}
N. Fusco, V. Millot, and M. Morini,
\emph{A quantitative isoperimetric inequality for fractional perimeters},
J. Funct. Anal. \textbf{261} (2011), 697--715.

\bibitem{Gonzalo}
J. Gonzalo,
\emph{Soap bubbles and isoperimetric regions in the product of a closed manifold with Euclidean space},
preprint, arXiv:1312.6311 (2013).

\bibitem{Grigoryan}
A. Grigor'yan,
\emph{Heat kernel and analysis on manifolds},
AMS/IP Studies in Advanced Mathematics, vol. 47, American Mathematical Society, Providence, RI; International Press, Boston, MA, 2009.

\bibitem{Riesz}
F. Riesz,
\emph{Sur une in\'egalit\'e int\'egrale},
J. London Math. Soc. \textbf{5} (1930), no. 3, 162--168.

\bibitem{RV}
M. Ritor\'e and E. Vernadakis,
\emph{Large isoperimetric regions in the product of a compact manifold with Euclidean space},
Adv. Math. \textbf{306} (2017), 958--972.

\bibitem{Seo}
J. Seo,
\emph{Sufficiently wide strips uniquely minimize the planar periodic fractional perimeter},
preprint, arXiv:2609.32588 (2026).

\end{thebibliography}
\end{document}